\documentclass[12pt]{article}
\usepackage[utf8]{inputenc}
\usepackage[T1]{fontenc}
\usepackage{amsmath,amssymb,amsthm}
\usepackage{geometry}
\usepackage{hyperref}
\usepackage{indentfirst}
\usepackage{booktabs}
\usepackage{xcolor}
\usepackage{fvextra}
\DefineVerbatimEnvironment{code}{Verbatim}{
  breaklines=true,
  breakanywhere=true,
  fontsize=\small,
  xleftmargin=1em,
}
\DefineVerbatimEnvironment{term}{Verbatim}{
  breaklines=true,
  breakanywhere=true,
  fontsize=\small,
  xleftmargin=1em,
}

\newtheorem{theorem}{Theorem}
\newtheorem{proposition}{Proposition}
\newtheorem{lemma}{Lemma}
\newtheorem{corollary}{Corollary}
\theoremstyle{definition}
\newtheorem{definition}{Definition}
\theoremstyle{remark}
\newtheorem{remark}{Remark}

\title{Genus-One Fibrations and the Jacobian of Linear Slices
in the Quintic Equal-Sum Problem}
\author{Valery Asiryan\\[3pt]
\small \texttt{asiryanvalery@gmail.com}}
\date{\small March 15, 2026}

\begin{document}

\maketitle

\begin{abstract}
We study the Diophantine equation
\[
a^5+b^5=c^5+d^5,\qquad a,b,c,d\in\mathbb{Z}_{\ge0},
\]
under the linear slicing constraint
\[
(c+d)-(a+b)=h.
\]
First, we give a self-contained proof that any solution must satisfy the necessary congruence
$30\mid h$; this is the $k=5$ instance of the \emph{modular divisibility obstruction} (MDO).

Second, we symmetrize the problem by passing to sums and differences
$S=a+b$, $u=b-a$, $T=c+d=S+h$, $v=d-c$ and reduce the quintic equality to a
biquadratic equation in $v$. Writing $Z=v^2$, we obtain an explicit discriminant
$D_Z(S,u)$, which must be a perfect square for solvability, and we
give an exact (and computable) criterion for the existence of an \emph{integer}
$v$ \emph{of the required parity} $v\equiv T\pmod{2}$ (equivalently, for the existence of
integers $c,d$ with sum $T$) in terms of $D_Z$ and integrality/square conditions on $Z$.
We also record the additional size constraints on $u$ and $v$ needed to recover
solutions with $a,b,c,d\in\mathbb{Z}_{\ge0}$.

Third, we provide the appropriate geometric interpretation: for nonzero slices $h\neq 0$ the equation
$Y^2=D_Z(S,u)$ defines a family of genus-one curves over $\mathbb{Q}(S)$, which need
not admit a rational section; thus one must pass to the Jacobian fibration to speak
about Mordell--Weil rank. On the Jacobian side we isolate a uniform arithmetic obstruction:
for every nonzero slice parameter $h$ the Jacobian $E_h/\mathbb{Q}(S)$ carries a global
rational $2$-torsion section and the square class of the quadratic discriminant governing full rational $2$-torsion is determined by an
explicit quintic factor
\[
Q_5(S,h)=S^5+4hS^4+8h^2S^3+8h^3S^2+4h^4S+\frac45 h^5.
\]
Over the function field $\mathbb{Q}(S)$, full rational $2$-torsion would force $S\cdot Q_5(S,h)$ to be a square in $\mathbb{Q}(S)$; this never occurs, because $S\cdot Q_5(S,h)$ has odd valuation at $S=0$ whenever $h\neq 0$.
For rational specializations $S=S_0\in\mathbb{Q}^\times$, the square condition that $S_0\cdot Q_5(S_0,h)$ be a square in $\mathbb{Q}$ reduces by homogeneity to rational points on a universal genus-two
hyperelliptic curve independent of $h$; using a verified \textsc{Magma} computation
(rank bound $0$) we show that this curve has only the affine rational point
$(0,0)$ and the two points at infinity, hence no nonsingular rational specialization acquires additional rational $2$-torsion.

Finally, for the first admissible integer slice $h=30$ we compute the classical invariants $(I,J)$ of the associated
binary quartic and write down the Jacobian elliptic curve $E_{30}/\mathbb{Q}(S)$ explicitly.
We then exploit the presence of a global rational $2$-torsion point and apply the
injectivity criterion of Gusi\'c--Tadi\'c for the specialization homomorphism to obtain
a computationally verified upper bound
\[
\mathrm{rank}\,E_{30}(\mathbb{Q}(S))\le 1.
\]
To prove the matching lower bound, we pass to the universal Jacobian model, shift by the rational $2$-torsion point, and exhibit an explicit rational section in closed form.
Specializing at $S=12$ on the slice $h=30$, where the specialization homomorphism is injective, yields a rational point on the specialized curve with nonzero $Y$-coordinate; since the torsion subgroup of this specialized curve is $\mathbb{Z}/2\mathbb{Z}$, the point has infinite order.
Hence
\[
\mathrm{rank}\,E_h(\mathbb{Q}(S))=1
\qquad\text{for every }h\in\mathbb{Q}^\times.
\]
A simple homogeneity normalization in the slice parameter (setting $x=S/h$ and scaling the Weierstrass coordinates) shows that this exact generic rank statement holds uniformly for all $h\neq 0$.
Thus the Jacobian fibration on each fixed nonzero slice has Mordell--Weil rank exactly one, and any corresponding family of \emph{integral} solutions on admissible integer slices would additionally have to satisfy the integrality, parity, and size constraints arising in the reduction.



\end{abstract}

\section{Introduction}

The existence of nontrivial equalities of the form
\begin{equation}\label{eq:main}
a^5+b^5=c^5+d^5,\qquad a,b,c,d\in\mathbb{Z}_{\ge0},
\end{equation}
with $\{a,b\}\ne\{c,d\}$ is unknown (see, e.g.,~\cite{LanderParkinSelfridge1967,Browning2002,Guy2004}).
For quartic Diophantine equations, geometric methods can be highly effective; for instance,
Elkies~\cite{Elkies1988} produced solutions to the related equation $A^4+B^4+C^4=D^4$ by exploiting
an elliptic fibration on a diagonal quartic K3 surface. In contrast, the quintic surface underlying
\eqref{eq:main} (which is of general type) remains much less understood from the standpoint of explicit arithmetic
constructions and obstructions.

One systematic way to organize a search is to impose a linear constraint on the sums:
\begin{equation}\label{eq:slice}
(c+d)-(a+b)=h,
\end{equation}
where $h\in\mathbb{Z}$ is fixed. We refer to \eqref{eq:slice} as a \emph{linear slice}.

Throughout, we set
\[
S:=a+b,\qquad T:=c+d=S+h.
\]
We also use the differences
\[
u:=b-a,\qquad v:=d-c.
\]
Then
\[
a=\frac{S-u}{2},\quad b=\frac{S+u}{2},\qquad
c=\frac{T-v}{2},\quad d=\frac{T+v}{2},
\]
so integrality requires $u\equiv S\pmod{2}$ and $v\equiv T\pmod{2}$.

\medskip
\noindent
\textbf{Nonnegativity constraints.}
In the original problem \eqref{eq:main} we require $a,b,c,d\in\mathbb{Z}_{\ge0}$.
In terms of $(S,u)$ this is equivalent to the size condition
\[
|u|\le S \qquad (\text{together with } u\equiv S \!\!\pmod 2),
\]
and similarly, since $c,d\ge0$ forces $T\ge0$, we have
\[
|v|\le T \qquad (\text{together with } v\equiv T \!\!\pmod 2).
\]
Because \eqref{eq:main} is symmetric under exchanging $a$ and $b$ (and also $c$ and $d$),
one may, when searching for solutions, impose the normalization $0\le u\le S$ and $0\le v\le T$
without loss of generality (keeping track of the induced identifications).

\medskip
\noindent
\textbf{Rational vs.\ Integral constraints.}
Several intermediate steps below (discriminant curves, genus-one fibrations over $\mathbb{Q}(S)$, and Jacobian elliptic curves over $\mathbb{Q}(S)$)
concern \emph{rational} points and sections over a function field. These geometric objects do not encode the \emph{inequality} constraints
$|u|\le S$ and $|v|\le T$ required to recover solutions of \eqref{eq:main} in $\mathbb{Z}_{\ge0}$.
Whenever we translate back from $(S,u,T,v)$ to $(a,b,c,d)\in\mathbb{Z}_{\ge0}$, we therefore impose the nonnegativity (size) constraints,
together with the integrality and parity conditions, explicitly.

\medskip
\noindent
\textbf{Universality vs.\ Integrality.}
The Jacobian fibration associated with a fixed nonzero slice parameter $h$ is defined over the function field $\mathbb{Q}(S)$.
At the level of Jacobian geometry (e.g.\ Mordell--Weil rank over $\mathbb{Q}(S)$), the parameter $h\neq 0$ can be normalized away by the change of variables $x=S/h$ and a scaling of Weierstrass coordinates; see Subsection~\ref{subsec:jacobian-universality}.
In contrast, when one seeks \emph{integral} solutions of \eqref{eq:main} on a fixed integer slice \eqref{eq:slice}, the congruence condition $30\mid h$ and the integrality/parity/size constraints in $(S,u,T,v)$ depend on the chosen integer $h$ and cannot be removed by such a normalization.

\section{A necessary modular obstruction: $30\mid h$}

\begin{proposition}\label{prop:30divides}
If \eqref{eq:main} holds for some $a,b,c,d\in\mathbb{Z}_{\ge0}$ and $h=(c+d)-(a+b)$, then
\[
30 \mid h.
\]
\end{proposition}

\begin{proof}
Let $p\in\{2,3,5\}$. Since $5\equiv 1 \pmod{p-1}$, Fermat's little theorem implies
$x^5\equiv x\pmod p$ for all integers $x$ (including multiples of $p$).
Reducing \eqref{eq:main} modulo $p$ gives
\[
a+b \equiv c+d \pmod p,
\]
hence $p\mid (c+d)-(a+b)=h$. Therefore $2\mid h$, $3\mid h$, and $5\mid h$, so
$30\mid h$.
\end{proof}

\begin{remark}
Thus only one residue class modulo $30$ is admissible for $h$, i.e.\ $29/30\approx 96.67\%$ of linear slices are ruled out by a simple congruence. In the language of prior work~\cite{AsiryanMDO} this is the $k=5$ case of the modular divisibility obstruction (MDO) $M_k\mid h$.
\end{remark}

In the remainder of the paper, when we discuss integral solutions on a fixed integer slice, we assume $30\mid h$. Before specializing to a particular slice, we isolate a uniform obstruction on the Jacobian side (Subsection~\ref{subsec:universal-obstruction}) showing that, for $h\neq 0$, the Jacobian cannot have full rational $2$-torsion.
We then discuss a normalization showing that the Jacobian geometry for any fixed $h\neq 0$ is universal up to scaling (Subsection~\ref{subsec:jacobian-universality}).
After this general discussion we concentrate on the first admissible positive slice:
\[
h=30,\qquad T=S+30.
\]

\section{Symmetrization and reduction to a discriminant condition}

\subsection{Symmetrization of the quintic sums}

\begin{lemma}\label{lem:sym}
Let $S=a+b$ and $u=b-a$ with $u\equiv S\pmod{2}$. Then
\begin{equation}\label{eq:L}
16(a^5+b^5)=S^5+10S^3u^2+5Su^4.
\end{equation}
Similarly, if $T=c+d$ and $v=d-c$ with $v\equiv T\pmod{2}$, then
\begin{equation}\label{eq:R}
16(c^5+d^5)=T^5+10T^3v^2+5Tv^4.
\end{equation}
\end{lemma}

\begin{proof}
Write $a=(S-u)/2$, $b=(S+u)/2$. Then
\[
32(a^5+b^5)=(S-u)^5+(S+u)^5
= 2\left(S^5+10S^3u^2+5Su^4\right),
\]
since the odd powers of $u$ cancel. Dividing by $2$ gives
\[
16(a^5+b^5)=S^5+10S^3u^2+5Su^4,
\]
which is \eqref{eq:L}. The proof
of \eqref{eq:R} is identical with $(S,u)$ replaced by $(T,v)$.
\end{proof}

Define
\[
\mathcal{L}(S,u):=S^5+10S^3u^2+5Su^4.
\]
Then \eqref{eq:main} is equivalent to
\[
T^5+10T^3v^2+5Tv^4=\mathcal{L}(S,u),
\]
with the parity constraints $u\equiv S\pmod{2}$ and $v\equiv T\pmod{2}$.

\subsection{A biquadratic equation in $v$}

Rearranging the last equality gives a biquadratic equation:
\begin{equation}\label{eq:biquad}
5T v^4 + 10T^3 v^2 + \bigl(T^5-\mathcal{L}(S,u)\bigr)=0.
\end{equation}
Let $Z=v^2$. Then \eqref{eq:biquad} becomes a quadratic equation in $Z$:
\begin{equation}\label{eq:quadZ}
5T Z^2 + 10T^3 Z + \bigl(T^5-\mathcal{L}(S,u)\bigr)=0.
\end{equation}

\begin{proposition}[Discriminant formula]\label{prop:disc}
The discriminant of \eqref{eq:quadZ} is
\begin{equation}\label{eq:D}
D_Z(S,u)=80T^6+20T\,\mathcal{L}(S,u)
=80T^6+20T\bigl(S^5+10S^3u^2+5Su^4\bigr).
\end{equation}
\end{proposition}

\begin{proof}
The discriminant of \eqref{eq:quadZ} equals
\[
(10T^3)^2 - 4\cdot (5T)\cdot\bigl(T^5-\mathcal{L}(S,u)\bigr)
=100T^6-20T(T^5-\mathcal{L}(S,u)),
\]
which simplifies to \eqref{eq:D}.
\end{proof}

\subsection{An exact integer-solvability criterion}

\begin{proposition}[Exact criterion for an integer $v$ (and the additional bounds for $c,d\ge0$)]\label{prop:exactcriterion}
Fix $h\in\mathbb{Z}$ and $S\in\mathbb{Z}_{\ge0}$ and set $T=S+h$.
Let $u\in\mathbb{Z}$ with $u\equiv S\pmod{2}$.
(When relating back to \eqref{eq:main} we additionally impose $|u|\le S$, which is equivalent to $a,b\in\mathbb{Z}_{\ge0}$.)
Assume $T\ne 0$.

\smallskip
\noindent
\textup{(a)} There exists an integer $v$ satisfying \eqref{eq:biquad} and the parity constraint
$v\equiv T\pmod{2}$ (equivalently, there exist integers $c,d\in\mathbb{Z}$ with $c+d=T$ and $v=d-c$)
if and only if there exists an integer $Y$ such that:
\begin{enumerate}
\item $Y^2=D_Z(S,u)$;
\item $Z:=\dfrac{-10T^3+Y}{10T}$ is a nonnegative integer;
\item $Z$ is a perfect square in $\mathbb{Z}$ (so that $v=\pm\sqrt{Z}\in\mathbb{Z}$);
\item $Z\equiv T\pmod{2}$ (equivalently, the resulting $v=\pm\sqrt{Z}$ satisfies $v\equiv T\pmod{2}$).
\end{enumerate}

\smallskip
\noindent
\textup{(b)} If in addition $T\ge 0$, then one can choose such a $v$ so that the recovered
\[
c=\frac{T-v}{2},\qquad d=\frac{T+v}{2}
\]
satisfy $c,d\in\mathbb{Z}_{\ge0}$ if and only if the conditions in \textup{(a)} hold and moreover
\begin{enumerate}
\setcounter{enumi}{4}
\item $Z\le T^2$ (equivalently, $|v|\le T$).
\end{enumerate}
\end{proposition}

\begin{proof}
Equation \eqref{eq:quadZ} has (rational) roots
\[
Z=\frac{-10T^3\pm \sqrt{D_Z(S,u)}}{10T}.
\]
Since $Y$ may be chosen with either sign once $Y^2=D_Z(S,u)$, it suffices to write
\[
Z=\frac{-10T^3+Y}{10T}.
\]
Hence an integer solution $Z\in\mathbb{Z}_{\ge0}$ exists if and only if $D_Z(S,u)$ is a square
$Y^2$ in \(\mathbb{Z}\) and $Z$ as above is an integer $\ge 0$.
Finally, $Z=v^2$ for an integer $v$ if and only if $Z$ is itself a perfect square.
This proves \textup{(a)}, except for the parity condition.

To recover integer $c,d$ from $T$ and $v$ via
$c=(T-v)/2$ and $d=(T+v)/2$, one must have $v\equiv T\pmod{2}$.
If $Z=v^2$ is a perfect square, then $v^2\equiv v\pmod{2}$, so the parity condition
$v\equiv T\pmod{2}$ is equivalent to $Z\equiv T\pmod{2}$. This yields the stated
equivalence in \textup{(a)}.

For \textup{(b)}, assume $T\ge0$. Then $c,d\in\mathbb{Z}_{\ge0}$ is equivalent to $|v|\le T$ (together with
$v\equiv T\pmod{2}$), since $c,d$ are the half-sum/half-difference of $T$ and $v$.
In terms of $Z=v^2$ this is exactly the bound $Z\le T^2$.
\end{proof}

\begin{remark}\label{rem:Tzero}
In the degenerate case $T=0$ (equivalently $S=-h$), the reduction above involves division by $T$.
Under the original nonnegativity constraints $a,b,c,d\in\mathbb{Z}_{\ge0}$, one necessarily has
$c+d=T=0$, hence $c=d=0$, and then \eqref{eq:main} forces $a^5+b^5=0$, so $a=b=0$.
Thus $T=0$ yields only the trivial solution, and the interesting case is $T\neq 0$.
\end{remark}

\begin{remark}
The condition ``$D_Z(S,u)$ is a square'' is only the \emph{first} filter: one must also
enforce the divisibility and congruence needed for integrality of $Z$, the parity constraint
$Z\equiv T\pmod{2}$ (equivalently $v\equiv T\pmod{2}$ to recover integer $c,d$), and finally the
requirement that $Z$ be a square.

Moreover, when relating back to the original Diophantine problem \eqref{eq:main} with
$a,b,c,d\in\mathbb{Z}_{\ge0}$, one must also impose the size constraints
$|u|\le S$ and (necessarily requiring $T\ge0$) $|v|\le T$, equivalently $Z\le T^2$.
These inequalities are essential to ensure nonnegativity of the recovered variables.
\end{remark}

\subsection{Algorithmic consequence}
For fixed $S$ and $T=S+h$, an exhaustive search over admissible $(u,v)$ with
$0\le u\le S$ and $0\le v\le T$ (and the corresponding parity conditions) is $O(S^2)$.
Proposition~\ref{prop:exactcriterion} reduces this to scanning only admissible $u$
values (about $S/2$ of them if one restricts to $0\le u\le S$ using the symmetry $a\leftrightarrow b$)
and performing a constant number of integer tests
(square tests for $D_Z$ and $Z$, the parity test $Z\equiv T\pmod{2}$, and in the nonnegative setting also the bound $Z\le T^2$),
i.e.\ $O(S)$ work per fixed $S$.

\section{Genus-one curves and Jacobian fibrations for admissible slices}

\subsection{The discriminant curve as a genus-one fibration}
Fix $h$ and write $T=S+h$. Over the rational function field $\mathbb{Q}(S)$, define the curve
\begin{equation}\label{eq:genusone}
\mathcal{C}_h:\qquad Y^2 = D_Z(S,u)=80T^6+20T\bigl(S^5+10S^3u^2+5Su^4\bigr).
\end{equation}
For generic $S$ (i.e.\ outside a finite set where the quartic in $u$ degenerates),
$\mathcal{C}_h$ is a smooth genus-one curve over $\mathbb{Q}(S)$.

\begin{remark}\label{rem:h0}
The slice $h=0$ is degenerate for \eqref{eq:genusone}: then $T=S$ and
\[
D_Z(S,u)=80S^6+20S\bigl(S^5+10S^3u^2+5Su^4\bigr)=100S^2(S^2+u^2)^2
\]
is a square in $\mathbb{Q}(S)[u]$. Accordingly, in the genus-one/Jacobian discussion we implicitly assume $h\neq 0$.
\end{remark}

\begin{remark}
A smooth genus-one curve over $\mathbb{Q}(S)$ need not have a $\mathbb{Q}(S)$-rational point.
Therefore it need \emph{not} be an elliptic curve in the strict sense (no chosen
rational basepoint, hence no group law on the curve itself).
Consequently, a computer algebra system failing to find an obvious rational point
(or failing to ``convert'' the model automatically) does \emph{not} imply that no
rational points exist and does \emph{not} determine any Mordell--Weil rank.
To speak about rank one must pass to the Jacobian fibration.
\end{remark}

\begin{remark}[Other admissible slices]\label{rem:otherslices}
The symmetrization and discriminant construction above applies verbatim for any
integer $h$ with $30\mid h$. For each such nonzero $h$ one obtains a genus-one fibration
$\mathcal{C}_h$ over $\mathbb{Q}(S)$ and a Jacobian elliptic curve
$E_h/\mathbb{Q}(S)$ defined by the invariants of the corresponding binary quartic.
Subsection~\ref{subsec:universal-obstruction} shows that for every nonzero $h$ the Jacobian admits a global rational
$2$-torsion section, but \emph{never} has full rational $2$-torsion over $\mathbb{Q}(S)$; moreover, no nonsingular rational specialization acquires additional rational $2$-torsion.
Moreover, after the normalization $x=S/h$ and a scaling of Weierstrass coordinates, the Jacobian fibrations for all fixed $h\neq 0$ become isomorphic over a rational function field; see Subsection~\ref{subsec:jacobian-universality}.
In particular, geometric invariants such as the generic Mordell--Weil rank over $\mathbb{Q}(S)$ are independent of the choice of $h\neq 0$.

In the remainder of the paper we restrict to the first nontrivial positive case $h=30$,
for which $h$ is admissible for integral solutions and the coefficients are relatively small, making explicit specialization computations over $\mathbb{Q}$ convenient.
For other admissible integer slices $h\in 30\mathbb{Z}\setminus\{0\}$, the Jacobian geometry is the same up to the normalization above, but the translation back to \emph{integral} solutions depends on the chosen $h$ through the integrality, parity, and size constraints in $(S,u,T,v)$.
\end{remark}

\subsection{A uniform discriminant factor and an obstruction to full rational $2$-torsion}\label{subsec:universal-obstruction}

In the binary quartic model \eqref{eq:genusone} the right-hand side is
\[
f_h(u)=a(S,h)\,u^4+c(S,h)\,u^2+e(S,h),
\]
with
\[
a(S,h)=100S(S+h),\qquad c(S,h)=200S^3(S+h),\qquad e(S,h)=80(S+h)^6+20(S+h)S^5.
\]
Since $f_h$ has no odd powers of $u$, its Jacobian admits a global rational $2$-torsion section.
To make this explicit, recall the invariants of a binary quartic
\[
f(u)=a u^4 + b u^3 + c u^2 + d u + e,\qquad
I = 12ae - 3bd + c^2,\qquad
J = 72ace + 9bcd - 27ad^2 - 27b^2e - 2c^3.
\]
and the Jacobian elliptic curve
\begin{equation}\label{eq:jac-general}
E_{I,J}:\qquad y^2 = x^3 - 27 I x - 27 J.
\end{equation}
(See, for example,~\cite{FisherBinaryQuartics,CremonaFisherStoll}.)
In our situation $b=d=0$, so $I=12ae+c^2$ and $J=72ace-2c^3$.

\begin{lemma}[Global $2$-torsion]\label{lem:e1-general}
Let $h$ be fixed and set $T=S+h$. Let $E_h/\mathbb{Q}(S)$ be the Jacobian of $\mathcal{C}_h$ written in the form \eqref{eq:jac-general}.
Then the cubic on the right-hand side has a root
\[
e_1(S,h)=-1200\,S^3T=-1200\,S^3(S+h)\in\mathbb{Q}(S),
\]
hence $(x,y)=(e_1(S,h),0)$ is a rational point of order $2$ on $E_h$.
\end{lemma}

\begin{proof}
This follows by a direct substitution of $x=e_1(S,h)$ into the cubic
$x^3-27I(S,h)x-27J(S,h)$ obtained from $a(S,h),c(S,h),e(S,h)$.
\end{proof}

Shifting the $x$-coordinate by $e_1(S,h)$, i.e.\ writing $x=X+e_1(S,h)$, puts $E_h$ into a standard $2$-torsion model
\begin{equation}\label{eq:2torsion-general}
E_h:\qquad y^2 = X\bigl(X^2 + A(S,h)\,X + B(S,h)\bigr),
\end{equation}
where
\[
A(S,h)=3e_1(S,h)=-3600\,S^3(S+h),\qquad B(S,h)=3e_1(S,h)^2-27I(S,h).
\]
The discriminant of the quadratic factor in \eqref{eq:2torsion-general},
\[
\Delta_{2}(S,h):=A(S,h)^2-4B(S,h),
\]
controls whether the remaining two $2$-torsion points are rational: the curve $E_h$ has \emph{full} rational $2$-torsion over the base field if and only if $\Delta_{2}(S,h)$ is a square.

\begin{definition}\label{def:Q5}
Define
\begin{equation}\label{eq:Q5}
Q_5(S,h)
:= S^5 + 4hS^4 + 8h^2S^3 + 8h^3S^2 + 4h^4S + \frac{4}{5}h^5,
\end{equation}
and let
\begin{equation}\label{eq:P}
P(S,h):=S\cdot Q_5(S,h).
\end{equation}
\end{definition}

\begin{proposition}[Explicit discriminant factor]\label{prop:disc-factor-general}
For every $h$, the quadratic discriminant $\Delta_{2}(S,h)$ factors as
\begin{equation}\label{eq:Delta2-factor}
\Delta_{2}(S,h)=12960000\,S\,(S+h)^2\,Q_5(S,h).
\end{equation}
In particular, since $12960000=3600^2$ and $(S+h)^2$ is already a square, the square class of $\Delta_{2}(S,h)$ is controlled by $P(S,h)=S\cdot Q_5(S,h)$.
\end{proposition}

\begin{proof}
Starting from $a(S,h),c(S,h),e(S,h)$ one computes $I(S,h)=12ae+c^2$ and $J(S,h)=72ace-2c^3$ and writes the Jacobian $E_h$ in the form \eqref{eq:jac-general}.
The shift $x=X+e_1(S,h)$ with $e_1(S,h)=-1200S^3(S+h)$ yields \eqref{eq:2torsion-general}.
A direct simplification of $A(S,h)^2-4B(S,h)$ then gives \eqref{eq:Delta2-factor}.
\end{proof}

\begin{proposition}[Generic non-squareness]\label{prop:P-not-square-functionfield}
Fix $h\in\mathbb{Q}^\times$. Then $P(S,h)=S\cdot Q_5(S,h)$ is not a square in $\mathbb{Q}(S)$.
Consequently, $E_h/\mathbb{Q}(S)$ does not have full rational $2$-torsion.
\end{proposition}

\begin{proof}
Since
\[
Q_5(0,h)=\frac45 h^5\neq 0
\qquad (h\neq 0),
\]
the rational function $P(S,h)=S\,Q_5(S,h)$ has valuation $1$ at the place $S=0$.
A square in $\mathbb{Q}(S)$ must have even valuation at every place, so $P(S,h)$ is not a square in $\mathbb{Q}(S)$.
By Proposition~\ref{prop:disc-factor-general}, full rational $2$-torsion would force $\Delta_2(S,h)$ to be a square in $\mathbb{Q}(S)$, equivalently $P(S,h)$ to be a square in $\mathbb{Q}(S)$, which is impossible.
\end{proof}

Consequently, a necessary condition for full rational $2$-torsion on a specialization $E_{h,S_0}/\mathbb{Q}$ (with $h\neq 0$ fixed and $S=S_0\in\mathbb{Q}^\times$) is that $P(S_0,h)$ be a square in $\mathbb{Q}$.
We now rule out this square condition on rational specializations uniformly in $h$.

\subsection{Reduction to a universal genus-two curve}\label{subsec:universal-genus2}

Assume $h\neq 0$ and let $S_0\in\mathbb{Q}^\times$. By homogeneity of \eqref{eq:P} we may set $x:=S_0/h$ and rewrite
\[
P(S_0,h)=h^6\Bigl(x^6+4x^5+8x^4+8x^3+4x^2+\frac45 x\Bigr).
\]
Thus $P(S_0,h)$ is a square in $\mathbb{Q}$ if and only if the genus-two curve
\[
y_1^2=x^6+4x^5+8x^4+8x^3+4x^2+\frac45 x
\]
has a rational point with $x\neq 0$. Clearing denominators via $Y:=5y_1$ gives the \emph{universal} hyperelliptic curve
\begin{equation}\label{eq:Cuniv}
\mathcal{C}_{\mathrm{univ}}:\qquad
Y^2 = 25x^6 + 100x^5 + 200x^4 + 200x^3 + 100x^2 + 20x,
\end{equation}
which is independent of $h$.

\subsection{Rational points on the universal curve}\label{subsec:universal-points}

\begin{proposition}\label{prop:Cuniv-Q}
The set $\mathcal{C}_{\mathrm{univ}}(\mathbb{Q})$ consists of the affine point $(0,0)$ and the two points at infinity.
\end{proposition}

\begin{proof}
We compute in \textsc{Magma} that the Jacobian $\mathrm{Jac}(\mathcal{C}_{\mathrm{univ}})$ has rank bound $0$ over $\mathbb{Q}$ and then apply the rank-$0$ Chabauty--Coleman routine~\cite{Chabauty1941,Coleman1985} to enumerate all rational points; see Section~\ref{subsec:magma-universal} for the script and console output.
\end{proof}

\begin{theorem}[No nontrivial squares]\label{thm:no-squares}
Let $S_0,h\in\mathbb{Q}$ with $S_0\neq 0$ and $h\neq 0$. Then $P(S_0,h)=S_0\cdot Q_5(S_0,h)$ is \emph{not} a square in $\mathbb{Q}$.
Equivalently, the equation $y^2=P(S_0,h)$ has no solutions with $y\in\mathbb{Q}$.
\end{theorem}

\begin{proof}
If $y^2=P(S_0,h)$ has a solution with $S_0,h\neq 0$, then the corresponding ratio $x=S_0/h$ yields a rational point on \eqref{eq:Cuniv} with $x\neq 0$.
By Proposition~\ref{prop:Cuniv-Q} no such point exists. Therefore $P(S_0,h)$ cannot be a square in $\mathbb{Q}$.
\end{proof}

\begin{corollary}\label{cor:no-full-2torsion}
Fix $h\in\mathbb{Q}^\times$.
\begin{enumerate}
\item The Jacobian $E_h/\mathbb{Q}(S)$ has a rational point of order $2$ but does \emph{not} have full rational $2$-torsion over $\mathbb{Q}(S)$.
\item For every $S_0\in\mathbb{Q}^\times$ such that the specialization
$E_{h,S_0}/\mathbb{Q}$ is nonsingular (equivalently, its elliptic discriminant is nonzero), the curve
$E_{h,S_0}$ has exactly one rational point of order $2$ (namely the specialization of $(e_1(S,h),0)$).
Equivalently, for every such $S_0$ the quadratic factor in the $2$-torsion model \eqref{eq:2torsion-general}
does not split over $\mathbb{Q}$.
\end{enumerate}

In particular, for every admissible nonzero integer slice parameter $h\in 30\mathbb{Z}\setminus\{0\}$ and every
$S_0\in\mathbb{Z}_{>0}$ with nonsingular specialization, the quadratic discriminant $\Delta_2(S_0,h)$ is not a square in $\mathbb{Q}$.
\end{corollary}

\begin{proof}
Part~\textup{(1)} follows from Lemma~\ref{lem:e1-general} and Proposition~\ref{prop:P-not-square-functionfield}.
For Part~\textup{(2)}, the specialization of $(e_1(S,h),0)$ gives at least one rational point of order $2$ on $E_{h,S_0}$.
If there were any additional rational $2$-torsion point, then the quadratic factor in \eqref{eq:2torsion-general}
would split over $\mathbb{Q}$, equivalently $\Delta_2(S_0,h)$ would be a square in $\mathbb{Q}$.
By Proposition~\ref{prop:disc-factor-general}, since $12960000$ and $(S_0+h)^2$ are squares, this would force $P(S_0,h)$ to be a square in $\mathbb{Q}$, contradicting Theorem~\ref{thm:no-squares}.
Because the specialization is nonsingular, the cubic on the right-hand side has distinct roots; with the quadratic factor irreducible over $\mathbb{Q}$, there is therefore exactly one rational point of order $2$.
The final assertion is immediate.
\end{proof}

\subsection{Normalization in the slice parameter and universality of the Jacobian fibration}\label{subsec:jacobian-universality}

The reduction in Subsection~\ref{subsec:universal-genus2} already uses the ratio $x=S/h$ to obtain a universal genus-two curve controlling the square class of $\Delta_2(S,h)$.
In fact, for every fixed $h\neq 0$ the \emph{entire} Jacobian fibration $E_h/\mathbb{Q}(S)$ is universal up to the same normalization.

\begin{remark}[Universality of the Jacobian geometry for $h\neq 0$]\label{rem:universal-jacobian}
Fix $h\in\mathbb{Q}^\times$ and set $T=S+h$.
A direct invariant computation (cf.\ Proposition~\ref{prop:IJ} below for the slice $h=30$) shows that the Jacobian of $\mathcal{C}_h$ admits the Weierstrass model
\begin{equation}\label{eq:JacExplicit-general}
E_h:\quad Y^2 = X^3 - 864000\,T^2\,S\bigl(3T^5+2S^5\bigr)\,X
- 345600000\,T^3\,S^4\bigl(9T^5+S^5\bigr),
\end{equation}
where $T=S+h$.
This equation is homogeneous in $(S,h)$, in the sense that under the substitution $S=hx$ (so $T=h(x+1)$) the coefficients of $X$ and the constant term scale as $h^8$ and $h^{12}$, respectively.
Thus, after the change of variables
\[
x=\frac{S}{h},\qquad X=h^4X',\qquad Y=h^6Y',
\]
the curve \eqref{eq:JacExplicit-general} becomes the \emph{universal} elliptic curve
\begin{equation}\label{eq:JacExplicit-univ}
E_{\mathrm{univ}}:\quad
Y'^2 = X'^3 - 864000\,x(x+1)^2\bigl(3(x+1)^5+2x^5\bigr)\,X'
- 345600000\,(x+1)^3x^4\bigl(9(x+1)^5+x^5\bigr)
\end{equation}
over $\mathbb{Q}(x)$, which is independent of $h$.
Since $\mathbb{Q}(S)\cong\mathbb{Q}(x)$ for fixed $h\neq 0$, it follows that the Jacobian geometries for all fixed $h\neq 0$ are isomorphic over a rational function field.
In particular, the Mordell--Weil rank $\mathrm{rank}\,E_h(\mathbb{Q}(S))$ is independent of $h\in\mathbb{Q}^\times$, and any bound proved for one convenient value (such as $h=30$) applies to all $h\neq 0$.

\smallskip
\noindent
\textup{Specializations.}
In Section~\ref{sec:magma} we verify injective specializations and compute ranks for the slice $h=30$ at several \emph{integer} values $S=S_0$.
Under the normalization above these correspond to rational values $x=x_0=S_0/30$ on the universal curve \eqref{eq:JacExplicit-univ} and suffice to bound the \emph{generic} Mordell--Weil rank over the function field (hence for all $h\neq 0$).
For a fixed \emph{integer} slice parameter $h\in 30\mathbb{Z}$, studying the arithmetic of \emph{integer} specializations $S=S_0\in\mathbb{Z}$ may still require separate injectivity checks, since the set of suitable integral parameters can depend on $h$.
\end{remark}

From now on we work on the first admissible positive slice $h=30$, so $T=S+30$.
This specialization is chosen to keep the coefficients small and to present convenient integer specializations; by Remark~\ref{rem:universal-jacobian} the Jacobian geometry for any fixed $h\neq 0$ is the same up to the normalization $x=S/h$ and a scaling of Weierstrass coordinates.
We write \eqref{eq:genusone} as
\[
Y^2 = a(S)\,u^4 + c(S)\,u^2 + e(S),
\]
where
\begin{align*}
a(S) &= 100S(S+30),\\
c(S) &= 200S^3(S+30),\\
e(S) &= 80(S+30)^6 + 20(S+30)S^5.
\end{align*}

\begin{proposition}\label{prop:infinity}
Over $\mathbb{Q}(S)$, the curve $\mathcal{C}_{30}$ has no rational points at infinity.
\end{proposition}

\begin{proof}
Consider the smooth projective model of $\mathcal{C}_{30}$ as a double cover of $\mathbb{P}^1_u$.
The points at infinity are precisely the points lying above $u=\infty$.

Set $x:=1/u$ and write the affine model near $x=0$.
From
\[
Y^2 = a(S)u^4 + c(S)u^2 + e(S)
\]
we obtain, after multiplying by $x^4$,
\[
(Yx^2)^2 = a(S) + c(S)x^2 + e(S)x^4.
\]
Hence the fiber above $x=0$ (i.e.\ above $u=\infty$) consists of two points defined over
$\mathbb{Q}(S)(\sqrt{a(S)})$; it is $\mathbb{Q}(S)$-rational if and only if $a(S)$ is a square in $\mathbb{Q}(S)$.
Since $a(S)=100S(S+30)=10^2\cdot S(S+30)$ and $S(S+30)$ is not a square in $\mathbb{Q}(S)$
(it has simple zeros at $S=0$ and $S=-30$), there are no $\mathbb{Q}(S)$-rational points above $u=\infty$.
\end{proof}

\subsection{The Jacobian elliptic curve via invariants}

For a binary quartic
\[
f(u)=a u^4 + b u^3 + c u^2 + d u + e,
\]
classical invariant theory defines invariants
\[
I = 12ae - 3bd + c^2,\qquad
J = 72ace + 9bcd - 27ad^2 - 27b^2e - 2c^3,
\]
and the Jacobian elliptic curve is
\begin{equation}\label{eq:jac}
E_{I,J}:\qquad Y^2 = X^3 - 27 I X - 27 J.
\end{equation}
See, e.g., standard references on 2-descent via binary quartics such as
\cite{FisherBinaryQuartics,CremonaFisherStoll}.

In our case $b=d=0$ and $f(u)=a(S)u^4 + c(S)u^2 + e(S)$, hence
\[
I(S)=12a(S)e(S)+c(S)^2,\qquad
J(S)=72a(S)c(S)e(S)-2c(S)^3.
\]

\begin{proposition}\label{prop:IJ}
For $h=30$ (so $T=S+30$) the invariants are
\begin{align}
I(S) &= 32000\,T^2\,S\bigl(3T^5+2S^5\bigr),\label{eq:Iexplicit}\\
J(S) &= 12800000\,T^3\,S^4\bigl(9T^5+S^5\bigr).\label{eq:Jexplicit}
\end{align}
Consequently, the Jacobian elliptic curve of $\mathcal{C}_{30}$ is
\begin{equation}\label{eq:JacExplicit}
E_{30}:\quad Y^2 = X^3 - 864000\,T^2\,S\bigl(3T^5+2S^5\bigr)\,X
- 345600000\,T^3\,S^4\bigl(9T^5+S^5\bigr),
\end{equation}
where $T=S+30$.
\end{proposition}

\begin{proof}
Using $b=d=0$, we have $I=12ae+c^2$ and $J=72ace-2c^3$.
Substituting $a(S)=100ST$, $c(S)=200S^3T$, $e(S)=80T^6+20TS^5$ and simplifying gives
\eqref{eq:Iexplicit} and \eqref{eq:Jexplicit}. Substituting these invariants into
\eqref{eq:jac} yields \eqref{eq:JacExplicit}.
\end{proof}

\begin{remark}
The computation in Proposition~\ref{prop:IJ} uses only the general coefficients $a(S,h)=100S(S+h)$, $c(S,h)=200S^3(S+h)$ and $e(S,h)=80(S+h)^6+20(S+h)S^5$ of the binary quartic model \eqref{eq:genusone}.
Replacing $30$ by an arbitrary $h\in\mathbb{Q}^\times$ yields the general Weierstrass equation \eqref{eq:JacExplicit-general} and hence the universal normalization described in Remark~\ref{rem:universal-jacobian}.
\end{remark}

\begin{remark}[What specialization theorems do and do not give]
If one computes the Mordell--Weil group $E_{30}(\mathbb{Q}(S))$ (the group of rational sections of the
Jacobian fibration), then Silverman's specialization theorem (see, for example,~\cite{SilvermanAEC})
implies that any section of
infinite order specializes to points of infinite order for all but finitely many specializations
$S=S_0$.
In particular, if $\mathrm{rank}\,E_{30}(\mathbb{Q}(S))>0$ then $\mathrm{rank}\,(E_{30})_{S_0}(\mathbb{Q})>0$ for infinitely many
(and in fact for ``almost all'') integers $S_0$ of good reduction.
However, the converse direction is false in general: knowing $\mathrm{rank}\,E_{30}(\mathbb{Q}(S))=0$ would not by itself
control the set of specializations with positive rank (``rank jumping'').
\end{remark}

\section{A rational 2-torsion point and a 2-torsion model}

In this section we exhibit a global rational $2$-torsion point on the Jacobian fibration
\eqref{eq:JacExplicit} and rewrite it in a standard $2$-torsion model.

Let $F_S(X)=X^3+a_4(S)X+a_6(S)$ be the cubic on the right-hand side of \eqref{eq:JacExplicit},
where
\[
a_4(S)=-864000\,T^2\,S\bigl(3T^5+2S^5\bigr),\qquad
a_6(S)=-345600000\,T^3\,S^4\bigl(9T^5+S^5\bigr).
\]

\begin{lemma}\label{lem:e1}
The cubic $F_S(X)$ has a root
\[
e_1(S)=-1200\,S^3T=-1200\,S^3(S+30)\in\mathbb{Q}(S).
\]
In particular, the point $(X,Y)=(e_1(S),0)$ is a rational $2$-torsion point on $E_{30}$.
\end{lemma}

\begin{proof}
A direct substitution of $X=e_1(S)$ into $F_S(X)$ shows that $F_S(e_1(S))=0$.
This identity was verified symbolically in \textsc{Magma} (see Section~\ref{sec:magma}).
Hence $(e_1(S),0)$ is a rational point of order $2$ on $E_{30}/\mathbb{Q}(S)$.
\end{proof}

Factoring the cubic gives
\[
X^3 + a_4 X + a_6 = (X-e_1)\bigl(X^2 + e_1 X + e_1^2 + a_4\bigr).
\]
If we shift $X$ by $e_1$ via $X = X'+e_1$, the equation becomes
\[
Y^2 = X' \bigl(X'^2 + A(S) X' + B(S)\bigr),
\]
where
\begin{equation}\label{eq:AandB}
A(S) = 3e_1(S),\qquad B(S) = 3e_1(S)^2 + a_4(S).
\end{equation}
In these coordinates the $2$-torsion point $(e_1(S),0)$ moves to $(X',Y)=(0,0)$.

A straightforward but somewhat lengthy computation in \textsc{Magma} yields the explicit
formulae
\begin{align}
A(S) &= -3600\,S^3T,\label{eq:Aexplicit}\\
B(S) &= -388800000\,S^7 - 46656000000\,S^6 - 2449440000000\,S^5 \nonumber\\
     &\qquad - 73483200000000\,S^4 - 1322697600000000\,S^3 \nonumber\\
     &\qquad - 13226976000000000\,S^2 - 56687040000000000\,S.\label{eq:Bexplicit}
\end{align}
Only the factorization of $B(S)$ and of the associated polynomial
\[
\Delta(S):=A(S)^2-4B(S),
\]
which is (up to powers of $B(S)$ and a nonzero constant) the nontrivial factor in the discriminant,
will be used in the sequel, and not the specific large integer coefficients.

\section{Injective specialization and an upper bound on the generic rank}

\subsection{The Gusi\'c--Tadi\'c injectivity criterion}

For an elliptic curve over a rational function field with a chosen $2$-torsion point,
Gusi\'c and Tadi\'c~\cite{GusicTadic} provide an explicit criterion for the injectivity
of the specialization homomorphism. We recall a slightly specialized version adapted to our setting.

Let $K=\mathbb{Q}(S)$ and consider an elliptic curve
\[
E_{30}:\quad Y^2 = X^3 + A(S)X^2 + B(S)X
\]
over $K$ with a rational $2$-torsion point at $(X,Y)=(0,0)$. Let $B(S),\Delta(S)$ denote
the polynomials in $S$ that occur in the Weierstrass model, and write $B^\mathrm{sf}$,
$\Delta^\mathrm{sf}$ for their squarefree parts in $\mathbb{Z}[S]$. Denote by
$\mathcal{H}$ the set of all nonconstant squarefree divisors (in $\mathbb{Z}[S]$) of either $B^\mathrm{sf}$ or $\Delta^\mathrm{sf}$.

\begin{theorem}[Gusi\'c--Tadi\'c, specialized form]\label{thm:GT}
Let $S_0\in\mathbb{Z}$ be such that the specialized curve $(E_{30})_{S_0}$ is nonsingular (equivalently,
its elliptic discriminant is nonzero).
If for every $g\in\mathcal{H}$ the value $g(S_0)$ is \emph{not} a square in $\mathbb{Q}$,
then the specialization homomorphism
\[
\sigma_{S_0}: E_{30}(\mathbb{Q}(S))\longrightarrow (E_{30})_{S_0}(\mathbb{Q})
\]
is injective.
\end{theorem}

This follows immediately from the main theorem of~\cite{GusicTadic}, since our set $\mathcal{H}$
contains, in particular, all irreducible factors of $B$ and of $A^2-4B$.

\subsection{Factorization of $B(S)$ and $\Delta(S)$}

Using \eqref{eq:Aexplicit} and \eqref{eq:Bexplicit}, we compute in \textsc{Magma}:

\begin{proposition}\label{prop:factors}
The polynomials $B(S)$ and $\Delta(S)=A(S)^2-4B(S)$ factor over $\mathbb{Z}[S]$ as
\begin{align*}
B(S) &= -388800000\,S(S+30)^2\bigl(S^4 + 60S^3 + 1800S^2 + 27000S + 162000\bigr),\\
\Delta(S) &= 12960000\,S(S+30)^2\bigl(S^5 + 120S^4 + 7200S^3 + 216000S^2 + 3240000S + 19440000\bigr).
\end{align*}
Consequently, up to multiplication by rational squares, the nonconstant factors of $B(S)$
and $\Delta(S)$ are:
\[
S,\quad S+30,\quad Q_4(S)=S^4 + 60S^3 + 1800S^2 + 27000S + 162000,
\]
\[
Q_5(S)=S^5 + 120S^4 + 7200S^3 + 216000S^2 + 3240000S + 19440000.
\]
The set $\mathcal{H}$ consists precisely of all nonconstant squarefree products of these
factors which divide either $S(S+30)Q_4(S)$ or $S(S+30)Q_5(S)$, i.e.\ the $11$ polynomials
\[
\begin{aligned}
& S,\ S+30,\ Q_4,\ Q_5,\ S(S+30),\ SQ_4,\ (S+30)Q_4,\\
& SQ_5,\ (S+30)Q_5,\ S(S+30)Q_4,\ S(S+30)Q_5.
\end{aligned}
\]
\end{proposition}

\begin{proof}
This is a straightforward factorization in $\mathbb{Z}[S]$, performed in \textsc{Magma};
see Section~\ref{sec:magma} for the corresponding script and output.
\end{proof}

\subsection{Injective specializations and ranks over $\mathbb{Q}$}

We now combine Theorem~\ref{thm:GT} with explicit computations for specialized curves
over $\mathbb{Q}$.

\begin{proposition}\label{prop:injectiveS}
There exist integers $S_0$ with $1\le S_0\le 100$ that satisfy the
Gusi\'c--Tadi\'c injectivity criterion.
In particular, each of the values
\[
S_0\in\{3,5,7,8,11,12,13,14,17,18,20,21\}
\]
has this property.
For every such $S_0$ the specialization homomorphism
\[
\sigma_{S_0}:E_{30}(\mathbb{Q}(S))\to (E_{30})_{S_0}(\mathbb{Q})
\]
is injective.
\end{proposition}

\begin{proof}
For each nonconstant squarefree divisor $g(S)$ of $B^\mathrm{sf}$ or $\Delta^\mathrm{sf}$, we
evaluated $g(S_0)$ for all integers $1\le S_0\le 100$ and tested whether $g(S_0)$ is a
square in $\mathbb{Q}$. We simultaneously excluded those $S_0$ for which the elliptic discriminant
of $(E_{30})_{S_0}$ vanishes. The \textsc{Magma} script and its output are shown in
Section~\ref{sec:magma}. In particular, each of the values
\[
S_0\in\{3,5,7,8,11,12,13,14,17,18,20,21\}
\]
satisfies the hypotheses of Theorem~\ref{thm:GT}, and hence the corresponding specialization
homomorphism is injective.
\end{proof}

\begin{remark}
We do not attempt here to classify all integers $S_0$ for which the specialization
homomorphism is injective; the subset singled out in Proposition~\ref{prop:injectiveS}
already suffices for our application to bounding the generic Mordell--Weil rank.
\end{remark}

For these same values $S_0$ we computed the Mordell--Weil rank of the specialized
curves $(E_{30})_{S_0}/\mathbb{Q}$ using standard \textsc{Magma} commands over number fields. We summarize the
results in Table~\ref{tab:ranks}.

\begin{table}[h]
\centering
\begin{tabular}{@{}ccc@{}}
\toprule
$S_0$ & $(E_{30})_{S_0}(\mathbb{Q})_{\mathrm{tors}}$ & $\mathrm{rank}\,(E_{30})_{S_0}(\mathbb{Q})$ \\ \midrule
$3$   & $\mathbb{Z}/2\mathbb{Z}$ & $2$ \\
$5$   & $\mathbb{Z}/2\mathbb{Z}$ & $2$ \\
$7$   & $\mathbb{Z}/2\mathbb{Z}$ & $4$ \\
$8$   & $\mathbb{Z}/2\mathbb{Z}$ & $3$ \\
$11$  & $\mathbb{Z}/2\mathbb{Z}$ & $3$ \\
$12$  & $\mathbb{Z}/2\mathbb{Z}$ & $1$ \\
$13$  & $\mathbb{Z}/2\mathbb{Z}$ & $3$ \\
$14$  & $\mathbb{Z}/2\mathbb{Z}$ & $1$ \\
$17$  & $\mathbb{Z}/2\mathbb{Z}$ & $1$ \\
$18$  & $\mathbb{Z}/2\mathbb{Z}$ & $2$ \\
$20$  & $\mathbb{Z}/2\mathbb{Z}$ & $3$ \\
$21$  & $\mathbb{Z}/2\mathbb{Z}$ & $2$ \\ \bottomrule
\end{tabular}
\caption{Ranks of the selected injective specializations $(E_{30})_{S_0}/\mathbb{Q}$ from Proposition~\ref{prop:injectiveS} (all with $1\le S_0\le 100$).}
\label{tab:ranks}
\end{table}

\subsection{An upper bound for the generic Mordell--Weil rank}

We now deduce a global constraint on the rank over $\mathbb{Q}(S)$.

\begin{theorem}\label{thm:rankbound}
Fix $h\in\mathbb{Q}^\times$ and let $E_h/\mathbb{Q}(S)$ be the Jacobian elliptic curve associated with the slice parameter $h\neq 0$.
Then
\[
\mathrm{rank}\,E_h(\mathbb{Q}(S)) \le 1.
\]
\end{theorem}

\begin{proof}
By Remark~\ref{rem:universal-jacobian}, the Mordell--Weil rank over the function field is independent of the choice of $h\neq 0$.
It therefore suffices to prove the stated bound for one convenient value, and we take $h=30$, i.e.\ the curve $E_{30}/\mathbb{Q}(S)$.

Let $\mathrm{rank}\,E_{30}(\mathbb{Q}(S))=r$. For each $S_0$ in
Proposition~\ref{prop:injectiveS}, the specialization map
\[
\sigma_{S_0}:E_{30}(\mathbb{Q}(S))\hookrightarrow (E_{30})_{S_0}(\mathbb{Q})
\]
is an injective group homomorphism. In particular,
\[
r\le \mathrm{rank}\,(E_{30})_{S_0}(\mathbb{Q})
\]
for all such $S_0$.
By Table~\ref{tab:ranks}, the minimal rank among these specializations is
$\mathrm{rank}\,(E_{30})_{12}(\mathbb{Q})=\mathrm{rank}\,(E_{30})_{14}(\mathbb{Q})
=\mathrm{rank}\,(E_{30})_{17}(\mathbb{Q})=1$.
Hence $r\le 1$.
\end{proof}

\begin{remark}
Theorem~\ref{thm:rankbound} provides the uniform upper bound $\mathrm{rank}\,E_h(\mathbb{Q}(S))\le 1$ for all $h\neq 0$.
In the next subsection we construct an explicit rational section on the universal model and show, by an injective specialization at $S_0=12$, that this section is non-torsion.
This gives the matching lower bound and hence the exact rank.
\end{remark}

\subsection{An explicit non-torsion section and exact generic rank}

For the lower bound it is convenient to work on the universal model from Remark~\ref{rem:universal-jacobian}.
To avoid confusion with the Weierstrass coordinate $x$, we write the base parameter as $t$ instead of $x$; thus $t=S/h$ and
\begin{equation}\label{eq:JacExplicit-univ-t}
E_{\mathrm{univ}}:\quad
Y'^2 = X'^3 - 864000\,t(t+1)^2\bigl(3(t+1)^5+2t^5\bigr)\,X'
- 345600000\,(t+1)^3t^4\bigl(9(t+1)^5+t^5\bigr).
\end{equation}

\begin{lemma}\label{lem:shifted-univ}
Under the change of variables
\[
X'=3600x-1200t^3(t+1),\qquad Y'=216000y,
\]
the universal model \eqref{eq:JacExplicit-univ-t} becomes
\begin{equation}\label{eq:shifted-univ}
\widetilde E_{\mathrm{univ}}:\qquad
y^2=x^3-t^3(t+1)x^2-\frac15\,t(t+1)^2g(t)\,x,
\end{equation}
where
\[
g(t)=5t^4+10t^3+10t^2+5t+1.
\]
\end{lemma}

\begin{proof}
This is a direct substitution and simplification.
\end{proof}

Define
\begin{align*}
g(t)&:=5t^4+10t^3+10t^2+5t+1,\\
Q(t)&:=5t^4+5t^3-t^2-3t-1,\\
R(t)&:=25t^4+55t^3+45t^2+15t+1,\\
P_{10}(t)&:=1250t^{10}+7250t^9+18625t^8+27950t^7+27125t^6+17750t^5\\
&\qquad\qquad\qquad\qquad +7885t^4+2310t^3+405t^2+30t-1.
\end{align*}

\begin{proposition}\label{prop:explicit-section}
The point
\[
\widetilde P(t):=
\left(
\frac{g(t)R(t)^2}{25Q(t)^2},
\frac{g(t)R(t)P_{10}(t)}{125Q(t)^3}
\right)
\]
lies on $\widetilde E_{\mathrm{univ}}(\mathbb{Q}(t))$.
Equivalently,
\[
P_{\mathrm{univ}}(t):=
\left(
-1200t^3(t+1)+144\,\frac{g(t)R(t)^2}{Q(t)^2},
1728\,\frac{g(t)R(t)P_{10}(t)}{Q(t)^3}
\right)
\]
is a rational point on $E_{\mathrm{univ}}(\mathbb{Q}(t))$.
\end{proposition}

\begin{proof}
Set
\[u(t):=\frac{R(t)}{5Q(t)},\qquad v(t):=\frac{P_{10}(t)}{25Q(t)^2}.
\]
A direct expansion and clearing of denominators gives the identity
\begin{equation}\label{eq:section-identity}
g(t)u(t)^4-t^3(t+1)u(t)^2-\frac15\,t(t+1)^2=v(t)^2.
\end{equation}
Multiplying \eqref{eq:section-identity} by $g(t)^2u(t)^2$ shows that
\[
x(t):=g(t)u(t)^2,\qquad y(t):=g(t)u(t)v(t)
\]
satisfy
\[
y(t)^2=x(t)^3-t^3(t+1)x(t)^2-\frac15\,t(t+1)^2g(t)\,x(t).
\]
Therefore $(x(t),y(t))=\widetilde P(t)$ lies on \eqref{eq:shifted-univ}.
The formula for $P_{\mathrm{univ}}(t)$ follows from the inverse change of variables in Lemma~\ref{lem:shifted-univ}.
\end{proof}

\begin{proposition}\label{prop:non-torsion-section}
The section $P_{\mathrm{univ}}(t)\in E_{\mathrm{univ}}(\mathbb{Q}(t))$ has infinite order.
\end{proposition}

\begin{proof}
Under the isomorphism of Remark~\ref{rem:universal-jacobian} with $h=30$, the universal section $P_{\mathrm{univ}}(t)$ corresponds to a section $P_{30}(S)\in E_{30}(\mathbb{Q}(S))$.
The specialization $S=12$ on the slice $h=30$ corresponds to $t=2/5$.
Substituting $t=2/5$ into Proposition~\ref{prop:explicit-section} gives
\[
P_{\mathrm{univ}}\!\left(\frac25\right)
=
\left(
\frac{101630663472}{1428025},
-\frac{6472868174296128}{341297975}
\right).
\]
Equivalently, after scaling back by $X=30^4X'$ and $Y=30^6Y'$ as in Remark~\ref{rem:universal-jacobian}, we obtain the rational point
\[
P_{12}:=
\left(
\frac{3292833496492800}{57121},
-\frac{188748835962475092480000}{13651919}
\right)
\in (E_{30})_{12}(\mathbb{Q}).
\]

By Proposition~\ref{prop:injectiveS}, the specialization homomorphism at $S_0=12$
\[
\sigma_{12}:E_{30}(\mathbb{Q}(S))\longrightarrow (E_{30})_{12}(\mathbb{Q})
\]
is injective.
It therefore suffices to show that $P_{12}$ is non-torsion.

From the computation recorded in Section~\ref{sec:magma} and Table~\ref{tab:ranks}, the torsion subgroup of $(E_{30})_{12}(\mathbb{Q})$ is $\mathbb{Z}/2\mathbb{Z}$.
Hence the only nonzero torsion point on $(E_{30})_{12}$ is its rational point of order $2$, which necessarily has $Y$-coordinate $0$.
Since the $Y$-coordinate of $P_{12}$ is nonzero, the point $P_{12}$ is not torsion.
Therefore the original section is non-torsion.
\end{proof}

\begin{theorem}\label{thm:rankexact}
For every $h\in\mathbb{Q}^\times$,
\[
\mathrm{rank}\,E_h(\mathbb{Q}(S))=1.
\]
\end{theorem}

\begin{proof}
By Proposition~\ref{prop:non-torsion-section}, the universal model $E_{\mathrm{univ}}/\mathbb{Q}(t)$ has a non-torsion rational section.
Transporting this section to any fixed $h\neq 0$ via the isomorphism in Remark~\ref{rem:universal-jacobian} shows that
\[
\mathrm{rank}\,E_h(\mathbb{Q}(S))\ge 1.
\]
On the other hand, Theorem~\ref{thm:rankbound} gives
\[
\mathrm{rank}\,E_h(\mathbb{Q}(S))\le 1
\]
for every $h\in\mathbb{Q}^\times$.
Therefore $\mathrm{rank}\,E_h(\mathbb{Q}(S))=1$ for all $h\neq 0$.
\end{proof}

\begin{remark}
Explicitly, for a fixed $h\in\mathbb{Q}^\times$, the non-torsion section on $E_h/\mathbb{Q}(S)$ is obtained from $P_{\mathrm{univ}}(t)$ by setting $t=S/h$ and applying the inverse scaling
\[
X=h^4X',\qquad Y=h^6Y'.
\]
Thus the lower bound $\mathrm{rank}\,E_h(\mathbb{Q}(S))\ge 1$ is completely explicit for every nonzero slice parameter.
\end{remark}

\section{Conclusions}

We provided a self-contained algebraic reduction of the slice problem for
the quintic equal-sum equation and highlighted the exact integrality constraints required
when using the discriminant method, including the necessary parity condition
$v\equiv T\pmod{2}$ to recover integer $c,d$, as well as the size constraints
$|u|\le S$ and $|v|\le T$ needed to recover solutions with $a,b,c,d\in\mathbb{Z}_{\ge0}$.

We also clarified the geometric interpretation: for every nonzero slice parameter $h\neq 0$ the discriminant equation defines a
family of genus-one curves, which need not possess a rational section, so generic-rank
claims require working with the Jacobian fibration.

On the Jacobian side, we isolated a uniform discriminant factor that governs the $2$-division field.
More precisely, for every nonzero slice parameter $h$ the Jacobian $E_h/\mathbb{Q}(S)$ admits a global rational
$2$-torsion section and, after passing to a standard $2$-torsion model, the quadratic discriminant
factors as
\[
\Delta_{2}(S,h)=12960000\,S\,(S+h)^2\,Q_5(S,h),\
Q_5(S,h)=S^5+4hS^4+8h^2S^3+8h^3S^2+4h^4S+\frac45 h^5.
\]
Over the function field $\mathbb{Q}(S)$, the factor $S\cdot Q_5(S,h)$ is not a square when $h\neq 0$,
because $Q_5(0,h)=\frac45 h^5\neq 0$, so $S\cdot Q_5(S,h)$ has odd valuation at $S=0$.
Hence $E_h/\mathbb{Q}(S)$ never has full rational $2$-torsion for $h\neq 0$.
For rational specializations $S=S_0\in\mathbb{Q}^\times$, the square condition that $P(S_0,h)=S_0\cdot Q_5(S_0,h)$ be a square in $\mathbb{Q}$
reduces by homogeneity to a universal genus-two hyperelliptic curve, and a verified \textsc{Magma}
computation (rank bound $0$) shows that no such nontrivial squares occur.
Thus, for every admissible nonzero integer slice parameter $h\in 30\mathbb{Z}\setminus\{0\}$ and every $S_0\in\mathbb{Q}^\times$ with nonsingular specialization,
the specialized Jacobian has exactly one rational $2$-torsion point.

For the first admissible positive slice $h=30$, we computed the Jacobian elliptic surface explicitly
via the invariants of the associated binary quartic, exhibited a rational $2$-torsion point,
and constructed a $2$-torsion model.
Using the injectivity criterion of Gusi\'c--Tadi\'c and explicit \textsc{Magma}
computations of specialized ranks over $\mathbb{Q}$, we first proved the universal upper bound
\[
\mathrm{rank}\,E_h(\mathbb{Q}(S))\le 1\qquad\text{for every }h\in\mathbb{Q}^\times.
\]
We then passed to the universal Jacobian model, shifted the rational $2$-torsion point to obtain the simplified equation \eqref{eq:shifted-univ}, and constructed the explicit rational section $P_{\mathrm{univ}}(t)$ of Proposition~\ref{prop:explicit-section}.
Specializing at $S=12$ on the slice $h=30$, where specialization is injective, yields a rational point on $(E_{30})_{12}(\mathbb{Q})$ with nonzero $Y$-coordinate.
Since the torsion subgroup of this specialized curve is $\mathbb{Z}/2\mathbb{Z}$, the specialized point has infinite order, so the section itself is non-torsion.
Consequently,
\[
\mathrm{rank}\,E_h(\mathbb{Q}(S))=1\qquad\text{for every }h\in\mathbb{Q}^\times,
\]
where the uniformity in $h\neq 0$ again follows from the normalization described in Remark~\ref{rem:universal-jacobian}.
Thus the Jacobian fibration on every fixed nonzero slice has Mordell--Weil rank exactly one.
This exact rank statement concerns the Jacobian fibration; by itself it still does not decide the existence (or nonexistence) of integral solutions of \eqref{eq:main} on a fixed integer slice, which additionally requires analyzing rational points on the corresponding genus-one torsors and the integrality, parity, and size constraints from the reduction.

As discussed in Remark~\ref{rem:otherslices}, the same symmetrization and genus-one
construction applies to every admissible slice with $30\mid h$, yielding a Jacobian
elliptic curve $E_h/\mathbb{Q}(S)$.
While full rational $2$-torsion never occurs over the function field for $h\neq 0$, and no nonsingular rational specialization has full rational $2$-torsion (Corollary~\ref{cor:no-full-2torsion}), the arithmetic problem of producing (or ruling out) integral solutions on a given integer slice still depends on $h$ through the integrality and size conditions on $(S,u,T,v)$ and through the arithmetic of the associated genus-one torsors.

Translating back to the original Diophantine equation, one sees that any putative
infinite family of \emph{integer} solutions with $(c+d)-(a+b)=h$ for a fixed nonzero admissible integer $h\in 30\mathbb{Z}\setminus\{0\}$ that is obtained via a rational section on the Jacobian side would necessarily be governed by a small Mordell--Weil group, and would be subject to additional square, parity, and size
constraints. Determining whether any such solutions exist remains an interesting open
problem.

\appendix
\section{Magma scripts and computational details}\label{sec:magma}

We record the \textsc{Magma}~\cite{Magma} scripts used to support the computations
described above. All scripts in this appendix are written for the slice $h=30$; by Remark~\ref{rem:universal-jacobian} this suffices to verify the generic Mordell--Weil rank bound for every fixed $h\neq 0$.

\subsection{Universal genus-two curve computation (rank $0$)}\label{subsec:magma-universal}

\noindent\textbf{Code (rank bound and rational points on $\mathcal{C}_{\mathrm{univ}}$).}
\begin{code}
Q := Rationals();
P<x> := PolynomialRing(Q);

// Polynomial with integer coefficients (cleared denominators by multiplying by 25)
// C_univ:  Y^2 = 25x^6 + 100x^5 + 200x^4 + 200x^3 + 100x^2 + 20x
Poly_Int := 25*x^6 + 100*x^5 + 200*x^4 + 200*x^3 + 100*x^2 + 20*x;

// Construct the curve and its Jacobian
C := HyperellipticCurve(Poly_Int);
J := Jacobian(C);

// Rank bound
print "Calculating Rank Bound...";
rb := RankBound(J);
print "Jacobian Rank Bound:", rb;

// If rank bound is 0, enumerate all rational points via the rank-0 Chabauty routine
if rb eq 0 then
    print "SUCCESS: Rank is 0. Computing ALL rational points via Chabauty0...";
    pts := Chabauty0(J);

    print "Rational Points on Curve (x, Y_new):";
    print pts;
    for pt in pts do
        // Points are printed in projective coordinates (x : Y : z)
        if pt[3] eq 0 then
            print "Point at Infinity found (corresponds to x = infinity in the projective closure).";
        else
            x_val := pt[1]/pt[3];
            if x_val eq 0 then
                print "Point x=0 found. (Corresponds to S=0; in the nonnegative setting this forces a=b=0.)";
            else
                printf "NON-TRIVIAL POINT FOUND: x = %o\n", x_val;
                print "This would imply a solution with S/h =", x_val;
            end if;
        end if;
    end for;
else
    print "Unexpected Rank > 0. Check calculations.";
end if;
\end{code}

\noindent\textbf{Transcript.}
\begin{term}
Calculating Rank Bound...
Jacobian Rank Bound: 0
SUCCESS: Rank is 0. Computing ALL rational points via Chabauty0...
Rational Points on Curve (x, Y_new):
{@ (0 : 0 : 1), (1 : -5 : 0), (1 : 5 : 0) @}
Point x=0 found. (Corresponds to S=0; in the nonnegative setting this forces
a=b=0.)
Point at Infinity found (corresponds to x = infinity in the projective closure).
Point at Infinity found (corresponds to x = infinity in the projective closure).
\end{term}

\subsection{Verification of the 2-torsion root}

\noindent\textbf{Code (Jacobian and $2$-torsion root).}
\begin{code}
Q := Rationals();
R<S> := PolynomialRing(Q);

T := S + 30;

// a4(S), a6(S) from the Jacobian model
a4 := -864000 * T^2 * S * (3*T^5 + 2*S^5);
a6 := -345600000 * T^3 * S^4 * (9*T^5 + S^5);

// define e1(S)
e1 := -1200 * S^3 * T;

// cubic F_S(X) over R = Q[S]
RX<X> := PolynomialRing(R);
FS := X^3 + a4*X + a6;

// Verify that e1 is a root
F_e1 := Evaluate(FS, e1);

if F_e1 eq 0 then
    print "e1(S) correctly satisfies F_S(e1)=0.";
else
    print "ERROR: e1(S) is not a root!";
    print F_e1;
end if;
\end{code}

\noindent\textbf{Transcript.}
\begin{term}
e1(S) correctly satisfies F_S(e1)=0.
\end{term}

\subsection{Construction and factorization of $B(S)$ and $\Delta(S)$}

\noindent\textbf{Code (construction of $A,B,\Delta$ and factorization).}
\begin{code}
Z := Integers();
R<S> := PolynomialRing(Z);

T := S + 30;

// a4(S), a6(S)
a4 := -864000 * T^2 * S * (3*T^5 + 2*S^5);
a6 := -345600000 * T^3 * S^4 * (9*T^5 + S^5);

// e1, A(S), B(S), Delta(S)
e1 := -1200 * S^3 * T;
A := 3*e1;
B := 3*e1^2 + a4;
Delta := A^2 - 4*B;

print "B(S) =", B;
print "Delta(S) =", Delta;

Factorization(B);
Factorization(Delta);
\end{code}

\noindent\textbf{Transcript.}
\begin{term}
B(S) = -388800000*S^7 - 46656000000*S^6 - 2449440000000*S^5 - 73483200000000*S^4
    - 1322697600000000*S^3 - 13226976000000000*S^2 - 56687040000000000*S
Delta(S) = 12960000*S^8 + 2332800000*S^7 + 198288000000*S^6 + 9797760000000*S^5
    + 293932800000000*S^4 + 5290790400000000*S^3 + 52907904000000000*S^2 +
    226748160000000000*S
[
    <2, 9>,
    <3, 5>,
    <5, 5>,
    <S, 1>,
    <S + 30, 2>,
    <S^4 + 60*S^3 + 1800*S^2 + 27000*S + 162000, 1>
]
[
    <2, 8>,
    <3, 4>,
    <5, 4>,
    <S, 1>,
    <S + 30, 2>,
    <S^5 + 120*S^4 + 7200*S^3 + 216000*S^2 + 3240000*S + 19440000, 1>
]
\end{term}

\subsection{Injective specializations (Gusi\'c--Tadi\'c criterion)}

\noindent\textbf{Code (injectivity test for $1\le S_0\le 100$).}
\begin{code}
Z := Integers();
R<S> := PolynomialRing(Z);

T := S + 30;

// a4(S), a6(S)
a4 := -864000 * T^2 * S * (3*T^5 + 2*S^5);
a6 := -345600000 * T^3 * S^4 * (9*T^5 + S^5);
// e1, A(S), B(S), Delta(S)
e1 := -1200 * S^3 * T;
A := 3*e1;
B := 3*e1^2 + a4;
Delta := A^2 - 4*B;

// normalisation and radicals
function NormPoly(f)
    g := PrimitivePart(f);
    if LeadingCoefficient(g) lt 0 then g := -g; end if;
    return g;
end function;

function Radical(f)
    g := SquarefreePart(f);
    g := NormPoly(g);
    return g;
end function;

function FactorList(f)
    rad := Radical(f);
    fac := [ NormPoly(ff[1]) : ff in Factorization(rad) | Degree(ff[1]) gt 0 ];
    return fac;
end function;

function SquarefreeDivisors(fac)
    divs := [];
    n := #fac;
    for mask in [1..2^n - 1] do
        g := R!1;
        for i in [1..n] do
            if ((mask div 2^(i-1)) mod 2) eq 1 then
                g *:= fac[i];
            end if;
        end for;
        g := NormPoly(g);
        Append(~divs, g);
    end for;
    return divs;
end function;

// Q-square test
function IsSquareQ(q)
    if q eq 0 then return true; end if;
    if q lt 0 then return false; end if;
    num := Integers()!Numerator(q);
    den := Integers()!Denominator(q);
    return IsSquare(num) and IsSquare(den);
end function;

// squarefree divisors
facB := FactorList(B);
facD := FactorList(Delta);

divsB := SquarefreeDivisors(facB);
divsD := SquarefreeDivisors(facD);

divs := [];
Skeys := {};
for g in divsB cat divsD do
    key := Sprint(g);
    if not key in Skeys then
        Include(~Skeys, key);
        Append(~divs, g);
    end if;
end for;

print "Total squarefree divisors =", #divs;

// injectivity test
function IsInjective(s0)
    for g in divs do
        v := Evaluate(g, s0);
        if IsSquareQ(Rationals()!v) then
            return false;
        end if;
    end for;
    // good reduction
    E0 := EllipticCurve([Evaluate(a4,s0), Evaluate(a6,s0)]);
    if Discriminant(E0) eq 0 then
        return false;
    end if;

    return true;
end function;

// search for 1 <= S0 <= 100
for s0 in [1..100] do
    if IsInjective(s0) then
        print "Injective specialization at S0 =", s0;
    end if;
end for;
\end{code}

\noindent\textbf{Transcript.}
\begin{term}
Total squarefree divisors = 11
Injective specialization at S0 = 3
Injective specialization at S0 = 5
Injective specialization at S0 = 7
Injective specialization at S0 = 8
Injective specialization at S0 = 11
Injective specialization at S0 = 12
Injective specialization at S0 = 13
Injective specialization at S0 = 14
Injective specialization at S0 = 17
Injective specialization at S0 = 18
Injective specialization at S0 = 20
Injective specialization at S0 = 21
\end{term}

\subsection{Rank computations for specialized curves}

\noindent\textbf{Code (ranks and torsion over $\mathbb{Q}$).}
\begin{code}
Q := Rationals();
R<S> := PolynomialRing(Q);

T := S + 30;

a4 := -864000 * T^2 * S * (3*T^5 + 2*S^5);
a6 := -345600000 * T^3 * S^4 * (9*T^5 + S^5);

Svals := [ 3, 5, 7, 8, 11, 12, 13, 14, 17, 18, 20, 21 ];

Eff := 2;

for s0 in Svals do
    a4_0 := Evaluate(a4, s0);
    a6_0 := Evaluate(a6, s0);

    E0 := EllipticCurve([a4_0, a6_0]);
    E0min := MinimalModel(E0);

    Tor := TorsionSubgroup(E0min);

    lb, ub := RankBounds(E0min : Effort := Eff);
    r0, exact := Rank(E0min : Effort := Eff);

    printf "S0 = %o\n", s0;
    printf "   torsion subgroup        = %o\n", Tor;
    printf "   RankBounds(E0min)       = (%o, %o)\n", lb, ub;
    printf "   Rank(E0min), exact?     = (%o, %o)\n\n", r0, exact;

    // verified exactness checks
    assert lb eq ub;
    assert exact;
    assert r0 eq lb;
end for;
\end{code}

\noindent\textbf{Transcript.}
\begin{term}
S0 = 3
   torsion subgroup        = Abelian Group isomorphic to Z/2
Defined on 1 generator
Relations:
2*Tor.1 = 0
   RankBounds(E0min)       = (2, 2)
   Rank(E0min), exact?     = (2, true)

S0 = 5
   torsion subgroup        = Abelian Group isomorphic to Z/2
Defined on 1 generator
Relations:
2*Tor.1 = 0
   RankBounds(E0min)       = (2, 2)
   Rank(E0min), exact?     = (2, true)

S0 = 7
   torsion subgroup        = Abelian Group isomorphic to Z/2
Defined on 1 generator
Relations:
2*Tor.1 = 0
   RankBounds(E0min)       = (4, 4)
   Rank(E0min), exact?     = (4, true)

S0 = 8
   torsion subgroup        = Abelian Group isomorphic to Z/2
Defined on 1 generator
Relations:
2*Tor.1 = 0
   RankBounds(E0min)       = (3, 3)
   Rank(E0min), exact?     = (3, true)

S0 = 11
   torsion subgroup        = Abelian Group isomorphic to Z/2
Defined on 1 generator
Relations:
2*Tor.1 = 0
   RankBounds(E0min)       = (3, 3)
   Rank(E0min), exact?     = (3, true)

S0 = 12
   torsion subgroup        = Abelian Group isomorphic to Z/2
Defined on 1 generator
Relations:
2*Tor.1 = 0
   RankBounds(E0min)       = (1, 1)
   Rank(E0min), exact?     = (1, true)

S0 = 13
   torsion subgroup        = Abelian Group isomorphic to Z/2
Defined on 1 generator
Relations:
2*Tor.1 = 0
   RankBounds(E0min)       = (3, 3)
   Rank(E0min), exact?     = (3, true)

S0 = 14
   torsion subgroup        = Abelian Group isomorphic to Z/2
Defined on 1 generator
Relations:
2*Tor.1 = 0
   RankBounds(E0min)       = (1, 1)
   Rank(E0min), exact?     = (1, true)

S0 = 17
   torsion subgroup        = Abelian Group isomorphic to Z/2
Defined on 1 generator
Relations:
2*Tor.1 = 0
   RankBounds(E0min)       = (1, 1)
   Rank(E0min), exact?     = (1, true)

S0 = 18
   torsion subgroup        = Abelian Group isomorphic to Z/2
Defined on 1 generator
Relations:
2*Tor.1 = 0
   RankBounds(E0min)       = (2, 2)
   Rank(E0min), exact?     = (2, true)

S0 = 20
   torsion subgroup        = Abelian Group isomorphic to Z/2
Defined on 1 generator
Relations:
2*Tor.1 = 0
   RankBounds(E0min)       = (3, 3)
   Rank(E0min), exact?     = (3, true)

S0 = 21
   torsion subgroup        = Abelian Group isomorphic to Z/2
Defined on 1 generator
Relations:
2*Tor.1 = 0
   RankBounds(E0min)       = (2, 2)
   Rank(E0min), exact?     = (2, true)
\end{term}

\subsection{Verification of the explicit section and a non-torsion specialization}

\noindent\textbf{Code (symbolic verification and non-torsion test).}
\begin{code}
Q := Rationals();
Qt<t> := FieldOfFractions(PolynomialRing(Q));

g := 5*t^4 + 10*t^3 + 10*t^2 + 5*t + 1;
Q_poly := 5*t^4 + 5*t^3 - t^2 - 3*t - 1;
R_poly := 25*t^4 + 55*t^3 + 45*t^2 + 15*t + 1;
P10 := 1250*t^10 + 7250*t^9 + 18625*t^8 + 27950*t^7 + 27125*t^6 +
       17750*t^5 + 7885*t^4 + 2310*t^3 + 405*t^2 + 30*t - 1;

u := R_poly / (5*Q_poly);
v := P10 / (25*Q_poly^2);

// 1. Symbolic verification of the quartic identity
LHS := g*u^4 - t^3*(t+1)*u^2 - (1/5)*t*(t+1)^2;
RHS := v^2;
assert LHS eq RHS;
print "Identity holds:", true;

// 2. Point on the Weierstrass model before the final coordinate change
x := g*u^2;
y := g*u*v;

// 3. Exact evaluation at t = 2/5
t0 := Q!2/5;
g0 := Evaluate(g, t0);
Q0 := Evaluate(Q_poly, t0);
R0 := Evaluate(R_poly, t0);
P10_0 := Evaluate(P10, t0);

assert Q0 ne 0;

// Specialized point on the simpler model
x0 := Evaluate(x, t0);
y0 := Evaluate(y, t0);

assert y0^2 eq x0^3 - t0^3*(t0+1)*x0^2 - (1/5)*t0*(t0+1)^2*g0*x0;
print "Point lies on the specialized curve:", true;

// 4. Point on the transformed universal model
X_univ := -1200*t0^3*(t0+1) + 144*(g0*R0^2)/(Q0^2);
Y_univ := 1728*(g0*R0*P10_0)/(Q0^3);

print "P_univ(2/5) X =", X_univ;
print "P_univ(2/5) Y =", Y_univ;

// 5. Check non-torsion on the specialized curve
E0 := EllipticCurve([0, -t0^3*(t0+1), 0, -(1/5)*t0*(t0+1)^2*g0, 0]);
P0 := E0![x0, y0];

T, phi := TorsionSubgroup(E0);
tors := { phi(tt) : tt in T };

print "Torsion subgroup =", T;
print "Is specialized point torsion?", P0 in tors;

// 6. Scale back to P_12 on E_30, if needed
X_12 := 30^4 * X_univ;
Y_12 := 30^6 * Y_univ;

print "P_12 X =", X_12;
print "P_12 Y =", Y_12;
\end{code}

\noindent\textbf{Transcript.}
\begin{term}
Identity holds: true
Point lies on the specialized curve: true
P_univ(2/5) X = 101630663472/1428025
P_univ(2/5) Y = -6472868174296128/341297975
Torsion subgroup =
Abelian Group isomorphic to Z/2
Defined on 1 generator
Relations:
    2*T.1 = 0
Is specialized point torsion? false
P_12 X = 3292833496492800/57121
P_12 Y = -188748835962475092480000/13651919
\end{term}

\end{document}